\documentclass[letterpaper,12pt,oneside]{article}

\usepackage[T1]{fontenc}
\usepackage[latin1]{inputenc}
\usepackage{epsfig}
\usepackage{amsmath}
\usepackage{amssymb}
\usepackage{mathtools}
\usepackage{amsthm}
\usepackage[active]{srcltx}
\usepackage{dsfont}
\usepackage[titles]{tocloft}
\usepackage{color}
\numberwithin{equation}{section}
\newtheorem{theoreme}{Theorem}[section]
\newtheorem{proposition}{Proposition}[section]
\newtheorem{remarque}[theoreme]{Remark}

\newtheorem{lemme}{Lemma}[section]
\newtheorem{ass}[theoreme]{Assumption}

\newtheorem{corollaire}[theoreme]{Corollary}
\newtheorem{definition}[theoreme]{Definition}
\newtheorem{notation}[theoreme]{Notation}
\definecolor{darkblue}{rgb}{0,0,0.7} 

\newcommand{\RR}{\ensuremath{\mathbb R}}

\newcommand{\PP}{\ensuremath{\mathbb P}}

\newcommand{\EE}{\ensuremath{\mathbb E}}
\newcommand{\NN}{\ensuremath{\mathbb N}}

\newcommand{\A}{\ensuremath{\mathcal A}}
\newcommand{\V}{\ensuremath{\mathcal V}}
\newcommand{\U}{\ensuremath{\mathcal U}}
\newcommand{\B}{\ensuremath{\mathcal B}}

\newcommand{\J}{\ensuremath{\mathcal J}}
\newcommand{\F}{\ensuremath{\mathcal F}}

\newcommand{\ind}{\ensuremath{\mathds 1}}
\newcommand{\G}{\ensuremath{\mathcal G}}

\newcommand{\xx}{\ensuremath{\mathbf{x}}}

\newcommand{\XX}{\ensuremath{\mathbf{X}}}

\newcommand{\om}{\omega}
\newcommand{\Om}{\Omega}
\newcommand{\De}{\Delta}
\newcommand{\la}{\lambda}

\newcommand{\de}{\delta}
\newcommand{\be}{\beta}
\newcommand{\ep}{\varepsilon}
\newcommand{\al}{\alpha}
\newcommand{\si}{\sigma}
\newcommand{\ga}{\gamma}

\newcommand{\VV}{\ensuremath{\mathbf{V}}}
\newcommand{\WW}{\ensuremath{\mathbf{W}}}

\newcommand{\vvv}{\ensuremath{\mathbf{v}}}
\newcommand{\uuu}{\ensuremath{\mathbf{u}}}

\newcommand{\limn}{\lim_{n\to\infty}}

\newcommand{\vs}{\vspace{-0.25cm}}

\title{Differential games with asymmetric \\ and correlated information}
\author{\textsc{Miquel Oliu-Barton}\\[0.25cm]
\small Institut de Math\'ematiques\\
\small Universit\'e de Neuch\^atel \\
\small  Rue Emilie-Argand 11, 2000 Neuch\^atel, Switzerland\\
\small miquel.oliu.barton@normalesup.org}
\date{Revised Version, January 2014}

\begin{document}
\maketitle \vs \vs \vs \vs \vs

\vspace{1cm}
\begin{small} 
\textbf{Abstract}. Differential games with asymmetric information were introduced by Cardaliaguet (2007). As in repeated games with lack of information on both sides (Aumann and Maschler (1995)), each player
receives a private signal (his type) before the game starts
and has a prior belief about his opponent's type. Then, a differential game is played in which
the dynamic and the payoff function depend on both types: each player is thus partially
informed about the differential game that is played.
The existence of the value function and some characterizations have been obtained under the assumption
that the signals are drawn independently.
In this paper, we drop this assumption and extend these two results to the general case of correlated types.
This result is then applied to repeated games with incomplete information: the characterization of the asymptotic value obtained by
Rosenberg and Sorin (2001) and Laraki (2001) for the independent case is extended to the general case.


\textbf{MSC2000 subject classification:} 49N30, 49N70, 91A05, 91A23, 93C41

\end{small}
\vspace{1cm}

\section{Introduction}\label{diffinco}
Differential games with incomplete information were introduced by Cardaliaguet \cite{carda07}, combining the aspects of classical differential games with the informational issues introduced by Aumann and Maschler \cite{AM95} in the study of repeated games. More precisely, the game consists in a dynamic
$$\dot{x}(t)=f(t,x(t),u(t),v(t)), \ t\in [t_0,T], \ x(t_0)=x_0,$$
where $u$ (resp. $v$) is the control played by the first (resp. second) player, and a family of payoff functions
$$ \J^{k\ell}(t_0,x_0)=\int_{t_0}^T \ga^{k\ell}(t,x(t),u(t),v(t))dt+g^{k\ell}(x(T)), \ (k,\ell)\in K\times L,$$
where $K$ and $L$ are two finite sets.
 Before the game starts, the type $k$ (resp. $\ell$) of player $1$ (resp. $2$)  is drawn according to a probability distribution $p$ (resp. $q$) over $K$ (resp. $L$). Each player knows (or is informed of) his type only. 
 Player $1$ (resp. $2$) aims at maximizing (resp. minimizing) the payoff. The main result of \cite{carda07} is that, under standard regularity assumptions of the dynamic and the payoff functions and under Isaacs' condition, the game has a value. Moreover, as a function of the initial time $t_0$, the initial position $x_0$ and of $p$ and $q$, the value is characterized as the unique dual solution of some Hamilton-Jacobi-Isaacs. The notion of dual solution is the following: the function $V(t_0,x_0,\, \cdot \,,\, \cdot\,)$ is concave-convex and Lipschitz continuous with respect to $(p,q)$, the concave conjugate of the value function with respect to $p$ is a sub-solution of the Hamilton-Jacobi-Isaacs equation, while the convex conjugate of the value function with respect to $q$ is a super-solution of the same equation.



The present paper is concerned with the same problem, dropping the (important) assumption that the players' types 
are drawn independently.
Considering correlated types is referred to as the \textit{dependent case} in the literature of repeated games, as opposed to the \textit{independent case}:  the pair of types is drawn according to some probability distribution over $K\times L$, which is not necessarily a direct product $p\otimes q$. Using the decomposition of a probability measure over the product set $K	\times L$, as the direct product of a marginal probability 
and a matrix of conditional probabilities, one obtains analogue results following the main lines of \cite{carda07}. These two components replace the couple of independent probabilities in the formulation of a new system of Hamilton-Jacobi-Isaacs equations that characterizes the value of the game. Our result is then applied to repeated games with incomplete information, leading to a characterization of the asymptotic value in the spirit of
Rosenberg and Sorin (2001) and Laraki (2001) (a pair of inequalities written only on the extreme points of the graph of the value function).

To underline the relevance of the dependent case we refer to Sorin and Zamir \cite{SZ85}: ''the dependent case is not only conceivable but seems to be rather the typical case : whenever the states of the world or the \textit{types} in question contain also the beliefs, as is typically the situation, the dependent case seems unavoidable''. However, our paper is the first in the literature on differential games with asymmetric information \cite{carda07,carda09b,carda09a,CR09, CR09b, CS10, souquiere10b, souquiere10a} to treat the dependent case.

The paper is organized as follows. In Section \ref{diffginco},
we introduce differential games with asymmetric information and define the sets of strategies (deterministic
and random). Then, we shortly describe some concepts from the theory of P.D.E's and from convex analysis.
Section \ref{sec_mainresults} is devoted to our main results.

\section{Differential games with asymmetric information}\label{diffginco}
A (standard, zero-sum) differential game is described by an initial state, a dynamic, a running payoff and a terminal payoff.
  A differential game with asymmetric information $\G(t_0,x_0,\pi)$
is described by two finite sets $K$ and $L$, a probability distribution $\pi\in \De(K\times L)$\footnote{For any finite set $X$, $\De(X):=\{ a:X\to [0,1],  \, \sum_{x\in X}a(x)=1\}$
denotes the set of probability distributions over $X$.
}, an initial time $t_0\in [0,1]$  and a family of differential games  indexed by $(k,\ell)\in K\times L$, with initial positions $x_0^{k\ell}$.
It is played as follows:
\begin{itemize}
\item First,  
a pair of parameters (or types) $(k,\ell)\in K\times L$ is drawn according to $\pi$: player $1$ is informed only about $k$, player $2$ only about $\ell$.
\item Second, the differential game $(x_0^{k\ell},f^{k\ell},\ga^{k\ell},g^{k\ell})$ 
is played on $[t_0,1]$, where
\begin{itemize}
\item $x^{k\ell}_0\in \RR^n$ is the initial state,
\item $f^{k\ell}:[0,1]\times \RR^n\times U \times V\to \RR^n$ is a dynamic,
\item $\ga^{k\ell}:[0,1]\times \RR^n\times U \times V\to \RR$, is a running payoff function, and
\item $g^{k\ell}:\RR^n\to \RR$ is a  terminal payoff function.
\end{itemize}
\end{itemize}
A crucial aspect of this model is the fact that, unlike standard differential games, none of the players knows
the true state of the world, i.e. each player is partially informed about the game that is being played.

A \emph{control} for player $1$ (resp. $2$) is a Lebesgue-measurable mapping from $[t_0,1]$ to $U$ (resp. $V$). Elements of $U$ and $V$ are identified with constant controls. The set of controls are denoted by $\U(t_0)$ and $\V(t_0)$ respectively.

The following assumptions on $f,\ga$ and $g$ are standard.
\begin{ass}\label{ass1JD}  For each $(k,\ell)\in K\times L$,
 \begin{itemize}
\item[] \vspace{-1.2cm}
\item[$(a)$] $f^{k\ell}$ and $\ga^{k\ell}$ are uniformly bounded,  uniformly Lipschitz in $(t,x)$ and continuous.
\item[$(b)$]  $g^{k\ell}$ is Lipschitz continuous and bounded.
\end{itemize}
\end{ass}
Assumption \ref{ass1JD} ensures that, for any pair of controls
$(\uuu,\vvv) \in \mathcal{U}(t_0)\times \mathcal{V}(t_0)$,  the following O.D.E. has a unique solution, modulo equality a.e.\footnote{a.e. is an abbreviation for almost everywhere.}
\begin{equation*} 
\begin{cases}
    \dot{\xx}^{k\ell}(t) =  f^{k\ell}(t,\xx(t),\uuu(t),\vvv(t)), \text{ a.e. on }  [t_0,1] \\
    \xx^{k\ell}(t_0) = x^{k\ell}_0.
    \end{cases}
\end{equation*}
Its solution is the \emph{trajectory} induced by the couple $(\uuu,\vvv)$. It belongs to
 $\mathcal{C}^1([t_0,1];\RR^n)$ and is denoted by $\xx^{k\ell}[t_0,x^{k\ell}_0,\uuu,\vvv]$.
The next assumption corresponds to perfect monitoring (i.e. observation of the past actions) in repeated games.

\begin{ass}\label{ass2JD} The players observe the past controls, i.e. at time $t\in [t_0,1]$,
$$(\uuu(s),\vvv(s))_{s\in [t_0,t]}$$
 is commonly known by the players.
\end{ass}

Both players know the description of the game. Player $1$ (resp. $2$) aims at maximizing (resp. minimizing) the following payoff functional:
$$\J^{k\ell}(t_0,x_0,\uuu,\vvv):=\int_{t_0}^1 \ga^{k\ell}(s,\xx^{k\ell}[t_0,x^{k\ell}_0,\uuu,\vvv](s),\uuu(s),\vvv(s)ds+g^{k\ell}\big(\xx^{k\ell}[t_0,x^{k\ell}_0,\uuu,\vvv](1)\big). $$


\begin{remarque} Standard differential games correspond to the case where both $K$ and $L$ are singletons, i.e. $|K|=|L|=1$.
Differential games with incomplete information on one side correspond to the case where either $|K|=1$ or $|L|=1$. The case where $\pi$ is a product measure  (i.e. there exist  $p\in \De(K)$ and $q\in \De(L)$ such that
$\pi=p\otimes q$) is known as the \textit{independent case}.
 \end{remarque}
 \subsection{Reduction}\label{reduc}
 Without loss of generality, the following simplification of the model is assumed:
\begin{itemize}
 \item[$(a)$] There is no running payoff, i.e. $\ga^{k\ell}\equiv 0$ for all $(k,\ell)\in K\times L$;
 \item[$(b)$] The dynamic and the initial position do not depend on the types, i.e. only the terminal payoff function is type dependent.
\end{itemize}
The differential game with asymmetric information $\G(t_0,x_0,\pi)$ is then described 
by following $5$-tuple $(t_0,x_0,f,g,\pi )$, where
$g=(g^{k\ell})_{k \ell}$ is a family of payoff functions indexed by $K$ and $L$.

Let us briefly explain why this reduction is possible (we refer the reader to the Appendix for more details). The past controls being commonly observed, at time $t$ both players can compute the $K\times L$ potential integral payoffs and positions induced by the pair of controls that have been played so far, i.e. in the interval $[t_0,t]$.
An auxiliary state variable in $(\RR\times \RR^n)^{K\times L}$, which includes this information, can thus be observed by both players. As a consequence, one can construct an auxiliary game satisfying $(a)$ and $(b)$, and which is strategically equivalent to the initial model.

\subsection{Strategies}
Let us define two sets of strategies: deterministic and random. In both cases, the definition of the strategy involves some partition of $[t_0,1]$: the choice of the partition is part of the strategy.
The main property of these sets is that any pair of strategies (and of a random event) determine a unique trajectory, and thus a unique outcome. The game is then said to be in normal form.
\subsubsection{Deterministic strategies}
\begin{definition}A map $\al:\V(t_0)\to \U(t_0)$ is a (deterministic) \emph{strategy} for player 1 if there exists a finite partition of $[t_0,1]$, $t_0=s_0<s_1<\cdots < s_N=1$, such that for all $\vvv_1,\vvv_2 \in\V(t_0)$ and $0\leq m<N$:
\[\vvv_1= \vvv_2 \text{ a.e. on } [s_0,s_m] \ \Longrightarrow \ \al(\vvv_1)= \al(\vvv_2)\text{ a.e. on }[s_0,s_{m+1}].\]
\end{definition}
Strategies are defined similarly for player $2$. Denote by $\A(t_0)$ (resp. $\B(t_0)$) the set of strategies of player $1$ (resp. $2$). As opposed to nonanticipative strategies, the following propety holds (\cite[Lemma 1]{CQ08}):

\textit{For any couple of strategies $(\al,\beta)\in \A(t_0)\times \B(t_0)$, there exists a unique pair $({\uuu},\vvv)\in \U(t_0)\times \V(t_0)$ such that
$\al(\vvv)=\uuu$ and $\be(\uuu)=\vvv$ a.e. on $[t_0,1]$.}

For any $(\al,\be)\in \A(t_0)\times \B(t_0)$, we denote by $\xx[t_0,x_0,\al,\be]\in \mathcal{C}^1([t_0,1];\RR^n)$ the trajectory induced by $\al$ and $\beta$, i.e. by the unique pair $(\uuu,\vvv)\in \U(t_0)\times \V(t_0)$ such that $\al(\vvv)=\uuu$ and $\be(\uuu)=\vvv$ a.e. on $[t_0,1]$.

\subsubsection{Random strategies}

The definition of random strategies involves a set $\mathcal{S}$ of (non trivial) probability spaces, which has to be stable by finite product. For simplicity, let 
$$\mathcal{S} = \{([0,1]^n, \B([0,1]^n), \mathcal{L}_n), \text{ for some } n\in \NN^*\}$$
where $ \B([0,1]^n)$ is the $\si$-algebra of Borel sets and $\mathcal{L}_n$ is the Lebesgue measure on $\RR^n$.
Endow the set of controls $\U(0)$ with the topology of the $L^1$-convergence, i.e. $\uuu_n$ converges to $\uuu$ if $\limn \int_0^1 d_U(\uuu_n(t),\uuu(t))dt=0$, where $d_U$ is the metric of $U$.
\begin{definition} A \emph{random} strategy for Player $1$ is a pair $((\Om_\al,\F_\al,\PP_\al),\al)$, where $(\Om_\al,\F_\al,\PP_\al)$ belongs to the set of probability spaces $\mathcal{S}$ and $\al:\Om_\al\times \V(t_0)\to \U(t_0)$ satisfies
\begin{itemize}
 \item $\al$ is a measurable function from $\Om_\al\times \V(t_0)$ to $\U(t_0)$, with $\Om_\al$ endowed with the $\si$-field $\F_\al$ and $\U(t_0)$ and $\V(t_0)$ with the Borel $\si$-field associated with the $L^1$ topology.
\item There exists a partition of $[t_0,1]$, $t_0=s_0<s_1<\dots<s_N=1$ such that, for any $0\leq m \leq N$, $\om \in \Om_\al$, and $\vvv_1, \vvv_2 \in \V(t_0)$:
$$\vvv_1=\vvv_2 \text{ a.e. on }[s_0,s_m] \Rightarrow \ \al(\om,\vvv_1)=\al(\om,\vvv_2) \text{ a.e. on }[s_0,s_{m+1}].$$
\end{itemize}
\end{definition}
Denote by $\A_r(t_0)$ the set
of random strategies for Player $1$. The set of random strategies for player $2$ is defined similarly, and is denoted by $\B_r(t_0)$.
\begin{notation}
 For simplicity, an element in $\A_r(t_0)$ is simply denoted by $\al$.
 The underlying probability space being always denoted by $(\Om_\al,\F_\al,\PP_\al)$.
 \end{notation}
As for the deterministic case, random 
strategies lead to a normal-form representation of the game (\cite{carda07}):
\begin{lemme}\label{fixrandom} For any pair $(\al,\beta)\in \A_r(t_0)\times \B_r(t_0)$ and any $\om=(\om_1,\om_2)\in \Om_\al \times \Om_\beta$, there exists a unique pair $(\uuu_\om,\vvv_\om)\in \U(t_0)\times \V(t_0)$ such that: 
\begin{equation}\label{fixrs}\al(\om_1,\vvv_\om)=\uuu_\om, \quad \text{ and }\quad \beta(\om_2,\uuu_\om)=\vvv_\om. \end{equation}
Moreover, the map $\rho:\Om_\al \times \Om_\be\to \U(t_0)\times \V(t_0)$,
$\om\mapsto (\uuu_\om,\vvv_\om)$ is measurable with respect to the $\F_\al\otimes \F_\be$ $\si$-field, and the topology of the $L^1$-convergence.
\end{lemme}
A direct consequence of Lemma \ref{fixrandom} is that to each pair $(\al,\beta)\in\A_r(t_0)\times \B_r(t_0)$ and to each event $\om\in\Om_\al\times \Om_\beta$ corresponds a unique trajectory, denoted by $\xx[t_0,x_0,\al(\om),\be(\om)]$.

A strategy of player $1$ in the game $ \G (t_0,x_0,\pi)$ is a vector of random strategies $\hat{\al}=(\hat{\al}^k)_{k\in K}$, where $\al^k\in \A_r(t_0)$ for each $k$.
Similarly, player $2$'s set of strategies is $\B_r(t_0)^L$.
\begin{remarque}
Random strategies contain deterministic ones, and the latter contain the set of controls, i.e. $\U(t_0)\subset \A(t_0)\subset \A_r(t_0)\subset \A_r(t_0)^K$.
\end{remarque}

Let us introduce some useful notation.
\begin{notation} For any pair of strategies $(\hat{\al},\hat{\beta})\in\A_r(t_0)^K\times \B_r(t_0)^L$ and measure
$\mu\in \De(K\times L)$, the expected payoff
is defined as follows:
$$\EE^{\mu}_{\hat{\al},\hat{\be}}\bigg[g^{k\ell}\big(
\XX_1^{t_0,x_0,\hat{\al}^k,\hat{\be}^\ell}(\om)\big)\bigg]=
 \sum_{K\times L} \mu^{k\ell}\int_{\Om_\al\times \Om_\beta}
g^{k\ell}\big(
\XX_1^{t_0,x_0,\hat{\al}^k,\hat{\be}^\ell}(\om)\big) d\PP_\al \otimes d\PP_\beta(\om),$$
where $\XX_t^{t_0,x_0,\hat{\al}^k,\hat{\be}^\ell}(\om):=\xx[t_0,x_0,\hat{\al}^k(\om),\hat{\beta}^\ell (\om)](t).$
This expectation makes sense: the map $(\uuu,\vvv)\mapsto \xx[t_0,x_0,\uuu,\vvv](t)$ is
continuous in the topology of the $L^1$-convergence, so that the maps $$\om\mapsto g^{k\ell}(\XX_1^{t_0,x_0,\hat{\al}^k,\hat{\be}^\ell}(\om)), \ (k,\ell)\in K\times L$$ are measurable in $\Om_\al\times \Om_\be$ and bounded. 

\end{notation}

\subsection{The upper and lower value functions}\label{value}
The upper and lower value functions $\VV^-, \VV^+:[0,1]\times \RR^n \times \De(K\times L)\to \RR$ are defined by
\begin{eqnarray*}
 \VV^-(t_0,{x}_0,\pi)&:= &{}{}\sup_{\hat{\al} \in \A_r(t_0)^K}\inf_{\hat{\beta}\in \B_r(t_0)^L}\EE^{\pi}_{\hat{\al},\hat{\be}}\bigg[g^{k\ell}\big(
\XX_1^{t_0,x_0,\hat{\al}^k,\hat{\be}^\ell}(\om)\big)\bigg],\\
 \VV^+(t_0,{x}_0,\pi)&:= &{}{}\inf_{\hat{\be}\in \B_r(t_0)^L}\sup_{\hat{\al} \in \A_r(t_0)^K}
\EE^{\pi}_{\hat{\al},\hat{\be}}\bigg[g^{k\ell}\big(
\XX_1^{t_0,x_0,\hat{\al}^k,\hat{\be}^\ell}(\om)\big)\bigg].
\end{eqnarray*}
The inequality $\VV^-\leq \VV^+$ holds everywhere.
The value exists if $\VV^-= \VV^+$, and we denote the common function by $\VV$.
Introduce the lower and upper Hamiltonians, $H^-,H^+:[0,1]\times \RR^n\times \RR^n\to \RR$ as follows: 
\begin{eqnarray*}
H^-(t,x,\xi)&:=&\sup_{u\in U}\inf_{v\in V}\,  \langle f(t,x,u,v),\xi\rangle,\\
H^+(t,x,\xi)&:=&\inf_{v\in V}\sup_{u\in U}\, \langle f(t,x,u,v),\xi\rangle.
\end{eqnarray*}
We are mostly concerned with the existence of the value function and its characterization.
Unlike standard differential games, where one can identify $\VV^-$ (resp. $\VV^+$) as the unique viscosity
solution of a (first-order) Hamilton-Jacobi-Isaacs equation with Hamiltonian $H^-$ (resp. $H^+$),
there is little hope in characterizing $\VV^-$ (resp. $\VV^+$) in the context of differential games with asymmetric information.
Rather, we will characterize the value function $\VV$, when it exists.
Isaacs' condition holds in the rest of the paper: \\

\textbf{Isaacs' condition.} $H^-(t,x,\xi)=H^+(t,x,\xi)$ for all $(t,x,\xi)\in [0,1]\times \RR^n\times \RR^n.$\\

We denote by $H$ the common Hamiltonian.
Cardaliaguet \cite{carda07} established the existence of the value function under Assumptions \ref{ass1JD}, \ref{ass2JD} and
Isaacs condition, in the case where $k$ and $\ell$ are drawn \emph{independently}.
The value function was characterized as the unique \emph{dual solution} of the following Hamilton-Jacobi-Isaacs equation:
\begin{equation}\label{HJI}
\partial_t w(t,x)+H(t,x,Dw(t,x))=0, \quad \text{on } (0,1)\times \RR^n.
\end{equation}
The definition of dual solutions involves the Fenchel conjugate and the notion of viscosity solutions introduced by Crandall and Lions \cite{CL83},
and used in the framework of differential games for the first time by Evans and Souganidis \cite{ES84}. Precisely, the notion of dual solution is the following: the function $V(t_0,x_0,\, \cdot \,,\, \cdot\,)$ is concave-convex and Lipschitz continuous with respect to $(p,q)$, the concave conjugate of the value function with respect to $p$ is a sub-solution of the Hamilton-Jacobi-Isaacs equation, while the convex conjugate of the value function with respect to $q$ is a super-solution of the same equation. The extension of this notion to the general, dependent case is left to Theorem \ref{main1}.


\subsection{Tools}
In this section, we start by defining a notion of convexity for functions defined in $\De(K\times L)$. Our definition goes back to Heuer \cite{heuer92} and is equivalent, yet easier to handle, to the notion of $I$-convexity given by Mertens and Zamir \cite{MZ71}.
Second, we recall the definition of viscosity solutions, and of some classical objects from convex analysis, such as the Fenchel conjugate and the sub-gradients.

\subsubsection{Convexity in $\De(K\times L)$}
For any probability measure $\mu \in \De(K\times L)$, let
$\mu^K\in \De(K)$ denote its marginal on $K$ (resp. $L$) and let $\mu^{L|K}\in \De(L)^K$ 
be the matrix of conditional probabilities, i.e.: 
$$\mu^K(k):=\sum_{\ell\in L} \mu(k,\ell)\ \text{ and }\
 \mu^{L|K}(\ell|k):=\frac{ \mu(k,\ell)}{\mu^K(k)}.$$
The probability $\mu$ is the direct product of $\mu^K$ and $\mu^{L|K}$, i.e.
$$\mu(k,\ell)=\mu^K(k)\mu^{L|K}(\ell|k), \quad \text{for all } (k,\ell)\in K\times L.$$
Similarly, $\mu=\mu^L\otimes \mu^{K|L}$, where $\mu^L\in \De(L)$ is the marginal on $L$ and $\mu^{K|L}\in \De(K)^L$ is the matrix of conditionals on $K$ given $\ell$.
\begin{notation} Let $\varphi:\De(K\times L)\to \RR$ be some map, and let $Q\in \De(L)^K$ be some matrix of conditional probabilities. 
We denote by $\varphi_K( \, \cdot \, , Q):\De(K)\to \RR$ 
the following function: 
\begin{equation*}
\varphi_K(p,Q):=\varphi(p\otimes Q), \quad \text{ for all } p\in \De(K).
\end{equation*} 
Similarly, one defines $\varphi_L(P,\, \cdot \,):\De(L)\to \RR$ for any $P\in \De(K)^L$ as follows:
\begin{equation*}
\varphi_L(P,q):=\varphi(q\otimes P), \quad \text{ for all } q\in \De(L).
\end{equation*} 

\end{notation}
\begin{definition} The map $\varphi:\De(K\times L)\to \RR$ is
\begin{itemize}
\item $K$-concave if $\varphi_K(\, \cdot \, , Q)$ is concave on $K$, for all $Q\in \De(L)^K$;
\item $L$-convex if $\varphi_L(P, \, \cdot \,)$ is convex on $L$, for all $P\in \De(K)^L$;
\end{itemize}
\end{definition}




\subsubsection*{Fenchel duality}
\begin{definition} 
For any $\varphi:\RR^n\to [-\infty,+\infty]$, the \emph{Fenchel transform} of $\varphi$, denoted by $\varphi^*$ is defined by
$\varphi^*:\RR^n\to [-\infty,+\infty], \ \varphi^*(x^*)=\sup_{x\in \RR^n} \langle x,x^*\rangle - \varphi(x).$
\end{definition}
Here, we define two slightly different transforms which are more convenient in the framework of games with incomplete information. The reason is that they correspond to the dual operators, one for each player, in the general theory of duality for games with incomplete information (see \cite[Section 4.6]{sorin02}).
\begin{definition} Let $\varphi:\RR^n\to \RR$.
Define its upper and lower conjugates $\varphi^\sharp, \varphi^\flat:\RR^n\to \RR$ as follows:
\begin{eqnarray*}\varphi^\sharp(x)&:=&
\sup_{y\in \RR^n}\varphi(y)-\langle y, x\rangle,\\
\varphi^\flat(y)&:=&
\inf_{x\in \RR^n} \varphi(x)+\langle x,y \rangle.
\end{eqnarray*}
\end{definition}
\noindent For all  $x,y \in \RR^n$ 
the following relations are straightforward:
\[ \varphi^\sharp(x)=(-\varphi)^{*}(-x),\quad  \text{and}\quad \varphi^\flat(y)=-\varphi^*(-y).\]

\subsubsection*{Sub-gradients}
\begin{definition}
 For any $\varphi:\RR^n\to [-\infty,+\infty]$ and $x\in \RR^n$, the \emph{sub-differential} 
  of $\varphi$ at $x$ is defined as follows: 
\begin{eqnarray*}\partial^-\varphi(x)&:=& \{x^*\in \RR^n \, | \, \varphi(x)+\langle x^*, y-x\rangle \leq \varphi(y), \  \forall y\in \RR^n \}. 
\end{eqnarray*}
\end{definition}
The super-differential $\partial^+ \varphi(x)$ is defined similarly.
The following result can be found in \cite[Section 12]{rockafellar97}.
 \begin{theoreme}[Fenchel equality]
 \label{fenchel2} Let $\varphi:\RR^n\to \RR\cup \{+\infty\}$ be convex and proper. Then $x^*\in \partial^-\varphi(x)$ if and only if  $\varphi^*(x^*)+\varphi(x)=\langle x,x^*\rangle.$
\end{theoreme}
Next, let us state a useful, straightforward lemma which follows directly from Fenchel equality and
the definitions of $\partial^-$ and $\partial^+$.
\begin{notation} Without further mention, functions defined on $\De(K)$ (resp. $\De(L)$) are extended to $\RR^K$ (resp. $\RR^L$) by $-\infty$ (resp. $+\infty$) in $\RR^K\backslash \De(K)$ (resp. $\RR^L\backslash \De(L)$). The sub-differentials (resp. super-differentials) of $f$ are defined according to this extension. %
\end{notation}
\begin{lemme}\label{fenchi} Let $\varphi:\De(K)\to \RR$  be a concave function and let $x\in \partial^+ \varphi(p)$. Then,
$$\varphi^\sharp(x)=\varphi(p)-\langle x,p\rangle \geq \varphi(p')-\langle x,p'\rangle, \quad \forall p'\in \De(K).$$
Similarly, if $\phi:\De(L)\to \RR$ is convex and
$y\in \partial^-\phi(q)$, then
$$\phi^\flat(-y)=\phi(q)+\langle -y,q\rangle \leq \phi(q')+\langle -y,q'\rangle, \quad \forall q'\in \De(L).$$
\end{lemme}
\noindent Note that the scalar products in Lemma \ref{fenchi} are in $\RR^K$ and $\RR^L$, respectively.
\begin{definition} Given $\varphi:\De(K)\to \RR$, let $\mathcal{E}_\varphi$ denote the set of extreme points of $\varphi$ on $\De(K)$. Explicitly, $p\in \mathcal{E}_\varphi$ if the equality $(p,\varphi(p))=\sum_{r\in R} \la_r (p_r,\varphi(p_r))$ with $R$ finite and $\la\in \De(R)$, $\la \gg0$ and $p_r\in \De(K)$ implies $p_r=p$ for all $r\in R$.

\end{definition}
\begin{lemme}\label{envelope} Let $\varphi:\De(K)\to \RR$  be a concave function. Then $\varphi^\sharp$ is differentiable at $x_0\in \RR^K$ if and only if $p_0:=-\nabla \varphi^\sharp(x_0)\in \De(K)$ is an extreme point of $\varphi$.
\end{lemme}
\begin{proof} It follows directly form the Envelope Theorem.
\end{proof}

\subsubsection*{Viscosity solutions}
\begin{definition}[Viscosity solutions]
 A map $w:[0,1]\times \RR^n\to \RR$ is a
\begin{itemize}
\item \emph{(viscosity) super-solution} of \eqref{HJI} if it is lower-semi-continuous in $(0,1)\times \RR^n$ and if, for any test function $\varphi\in \mathcal{C}^1([0,1]\times \RR^n;\RR)$ such that $w-\varphi$ has a local minimum at some point $(t,x)\in (0,1)\times \RR^n$, one has:
$$ \partial_t \varphi(t,x)+H(t,x,D\varphi(t,x))\leq 0.$$
\item \emph{(viscosity) sub-solution} of \eqref{HJI} if it is upper-semi-continuous in $(0,1)\times \RR^n$ and if, for any test function $\phi\in \mathcal{C}^1([0,1]\times \RR^n;\RR)$ such that $w-\phi$ has a local maximum at some point $(t,x)\in (0,1)\times \RR^n$, one has:
$$ \partial_t \phi(t,x)+H(t,x,D\phi(t,x))\geq 0.$$
\item \emph{viscosity solution} of \eqref{HJI} if it is both a super-solution and a sub-solution.\end{itemize}
\end{definition}
Three basic properties of \emph{viscosity solutions} are existence, uniqueness and stability with respect to uniform convergence.

\section{Main results}\label{sec_mainresults} 
In this section we state and prove our main result: the existence and characterization of the value function for
differential games with asymmetric and correlated information.

We can now state our main result. In the sequel, {sub-solutions} (resp. {super-solutions}) refer to \textit{viscosity sub-solutions} (resp. \textit{super-solutions}).

\begin{theoreme}[Existence and characterization of the value $\VV$]\label{main1} Assume Isaacs condition. Then, the value exists and is the unique Lipschitz continuous function on $[0,1]\times \RR^n\times \De(K\times L)$ satisfying:
\begin{itemize}
  \item[-] $p\mapsto \VV_K(t,x,p,Q)$ is concave for all $(t,{x},Q)\in [0,1]\times \RR^n \times \De(L)^K$,
  \item[-] $q\mapsto \VV_L(t,{x},P,q)$ is convex for all $(t,{x},P)\in [0,1]\times \RR^n\times \De(K)^L$,
 \item[-] For all $(\zeta,Q)\in \RR^K\times \De(L)^K$,
$(t,x)\mapsto\VV_K^\sharp(t,x,\zeta,Q)$ is a sub-solution of
\eqref{HJI}.
 \item[-] For all $(P,\eta)\in \De(K)^L \times \RR^L$,
$(t,{x})\mapsto \VV_L^{\flat}(t,{x},P,\eta)$ is a super-solution of
\eqref{HJI}.
\item[-] $\VV(1,{x},\pi)=\sum_{k,\ell}\pi^{k\ell} g^{k\ell}(x^{k\ell})$, for all $({x},\pi)\in \RR^n\times \De(K\times L)$.
\end{itemize}
\end{theoreme}
\begin{remarque}
 The characterization of the value function in Theorem \ref{main1} gives an explicit definition
 of what was called a \emph{dual solution} to the Hamilton-Jacobi equation \eqref{HJI} in \cite{carda07}.
\end{remarque}

\subsection{Proof of Theorem \ref{main1}}
We follow the main ideas in the proof of \cite{carda07}. Also, we use the duality techniques introduced by De Meyer \cite{dm96} for games with incomplete information on one side (see \cite[Chapter 2]{sorin02} for a general presentation), and extended in \cite{GOB13} to the case of general type dependence. The proof can be summarized as follows.\\
\indent \emph{Step 1}.
One proves the $K$-concavity, $L$-convexity and Lipschitz continuity of both the upper and the lower value functions. These results being classical, we have preferred to omit the proof. \\
\indent \emph{Step 2}. One  proves a sub-dynamic programming principle for $\VV_K^{+\sharp}$ at fixed
$(\zeta, Q)$. For that, we use an  alternative expression for $\VV_K^{+\sharp}(t,{x},\zeta,Q)$ which, again,
is a general property of any normal-form game with incomplete information. We then deduce that
$(t,x)\mapsto \VV^{+\sharp}(t,x,\zeta,Q)$ is a sub-solution of  \eqref{HJI}.
Symmetrical results hold for $\VV^{-\flat}$ by exchanging the roles of the players: $(t,x)\mapsto \VV^{-\flat}(t,x,P,\eta)$ is a super-solution of  \eqref{HJI}. \\
\indent \emph{Step 3}. One concludes using a new comparison theorem inspired by the analog result in \cite{carda07}.
Here, Assumption \ref{ass1JD} ensures that the Hamiltonian is regular enough.
\subsection{Regularity}
\begin{lemme}\label{regg} The upper and lower value function are $K$-concave, $L$-convex and Lipschitz continuous. 
 \end{lemme}
The $K$-concavity, $L$-convexity and the Lipschitz continuity with respect to $\pi\in \De(K\times L)$ is standard for games with incomplete information (see \cite[Chapter 2]{sorin02}). The proof is omitted, for there is nothing particular to our model. The regularity in $(t,x)\in [0,1]\times \RR^n$ follows from the assumptions on $f$, $\ga$ and $g$ (see, for instance, \cite{cardanotes}).
\subsection{Sub-dynamic programming principle}
\begin{proposition}\label{subprogr'}
For all $(\zeta,Q)\in \RR^K\times \De(L)^K$,
 $t,t+h\in [0,1]$ and $x\in \RR^n$: 
\begin{equation*}\label{subd2}\VV_K^{+\sharp}(t,{x},\zeta,Q)\leq {}{} \inf_{\beta\in \B(t)}\sup_{\al\in \A(t)} \VV_K^{+\sharp}(t+h,{\xx}[t,{x},\al,\beta](t+h),\zeta,Q).
\end{equation*}
\end{proposition}
\begin{proof}
Consider the following alternative expression for $\VV_K^{+\sharp}(t,{x},\zeta,Q)$, which is a general property of
normal-form games with convex sets of strategies
(see \cite[Chapter 2]{sorin02}):
\begin{equation}\label{alt}\VV_K^{+\sharp}(t,{x},\zeta,Q)= {}{}
 \inf_{\hat{\beta}\in \B_r(t)^L}\sup_{\al\in \A_r(t)} \max_{k\in K}
 \EE^{k,Q}_{\al, \hat{\beta}} \bigg[ g^{k\ell}\big(\xx[t,x,\al(\om),\be(\om)](1) \big) \bigg]-\zeta^k. 
\end{equation}

Recall that $\om=(\om_1,\om_2)\in \Om_\al\times \Om_\be$.
Using the convexity of the map $z\mapsto \max_{k\in K} z^k$, one can replace the random strategy of player $1$ by  a deterministic one.  Indeed, for any random strategy $\al\in \A_r(t)$, one has
\begin{eqnarray*}
  & \max_{k\in K} \EE^{k,Q}_{\al,\hat{\beta}}\bigg[g^{k\ell}\big(\xx[t,x,\al(\om),\be(\om)](1)\big)\bigg]-\zeta^k \\
 =& \max_{k\in K} \int_{\Om_\al} \EE^{k,Q}_{\hat{\beta}}\bigg[g^{k\ell}\big(\xx[t,x,\al(\om_1),\be(\om_2)](1)\big)\bigg]d\PP_\al(\om_1)-\zeta^k, \\
 \leq &
 \int_{\Om_\al} \max_{k\in K}\left( \EE^{k,Q}_{\hat{\beta}}
 \bigg[g^{k\ell}\big(\xx[t,x,\al(\om_1),\be(\om_2)](1)\big)\bigg]-\zeta^k\right) d\PP_\al(\om_1),\\
 \leq &
 \sup_{\om_1\in \Om_\al} \max_{k\in K}\EE^{k,Q}_{\hat{\beta}}
  \bigg[g^{k\ell}\big(\xx[t,x,\al(\om_1),\be(\om_2)](1)\big)\bigg]-\zeta^k.
\end{eqnarray*}
On the other hand, clearly $\sup_{\al\in \A_r(t)}\sup_{\om_1\in \Om_\al} =\sup_{\al\in \A(t)}$. 
It follows that 
\begin{equation}\label{alt3}\VV_K^{+\sharp}(t,{x},\zeta,Q)\leq {}{} \inf_{\hat{\beta}\in \B_r(t)^L}
\sup_{\al\in \A(t)} \max_{k\in K} \EE^{k,Q}_{\hat{\beta}}
\bigg[g^{k\ell}\big(\xx[t,x,\al,\be(\om_2)](1)\big)\bigg]-\zeta^k.
\end{equation}
For any $\hat{\be}\in \B_r(t)$, the regularity assumptions on $f$ and $g$ ensure the Lipschitz-continuity of the map
$$x\mapsto \sup_{\al\in A(t)} \max_{k\in K} \EE^{k,Q}_{\hat{\beta}}
\bigg[g^{k\ell}\big(\xx[t,x,\al,\be(\om_2)](1)\big)\bigg]-\zeta^k.$$
\begin{notation} We say that $\hat{\be}$ is  $\ep$-optimal for
$\VV_K^{+\sharp}(t,x,\zeta,Q)$ if it reaches the infimum in the formulation \eqref{alt3} up to $\ep$.
\end{notation}

For any $\ep>0$, let $\de>0$ be such that if $\hat{\beta}^y$ is $\ep$-optimal for
$\VV_K^{+\sharp}(t,y,\zeta,Q)$, for some $y\in \RR^n$, then $\hat{\beta}^y$ is $2\ep$-optimal for $\VV_K^{+\sharp}(t,y',\zeta,Q)$
for any $y'\in B(y,\de)$.

The set of reachable points at time $t+h$ is clearly
contained in $B(x,\|f\|)$. Let $(x_i)_{i \in I}$ be a finite family of points
such that $\bigcup_{i\in I}B(x_i,\de)$ covers $B(x,\|f\|)$, and let $E_i$ be a Borel
partition of $B(x,\|f\|)$ such that, for all $i \in I$, $E_i\subset B(x_i,\de)$.
We aim at proving \eqref{subd2} by explicitly constructing a strategy $\hat{\be}_\ep\in \B_r(t)$ for player $2$ which is $\ep$-optimal in \eqref{alt3}.
Intuitively, the strategy is as follows:
\begin{itemize}
\item[-] Play $\be_0\in \B(t)$ on $[t,t+h]$, where
$\be_0$ is $\ep$-optimal in the right-hand-side of \eqref{subd2}.
\item[-] If $\xx[t,x,\al,\be](t+h)$ belongs to
$E_i$, then play in the remaining of the game $[t+h,1]$
a strategy $\hat{\be}_{i}\in \B_r(t+h)$ which is $\ep$-optimal for $\VV_K^{+\sharp}(t+h,x_i,\zeta,Q)$.
\end{itemize} 
Let us define $\hat{\be}_\ep$ formally. For $i\in I$ and $\ell\in L$, let $(\Om_i^\ell,\F_i^\ell,\PP_i^\ell)=(\Om_{\hat{\be}_{i}^\ell},\F_{\hat{\be}_{i}^\ell},\PP_{\hat{\be}_{i}^\ell})$ be the probability space associated to $\hat{\beta}^\ell_i$,  and let $t+h=s_0<s_1<\dots <s_N=1$ be a common partition to all $\hat{\be}^\ell_i$. This is possible because $I\times L$ is finite.  For any $\ell\in L$, let
$$(\Om^\ell, \F^\ell,\PP^\ell)=\left(\prod\nolimits_{i\in I} \Om_i^\ell,
\otimes_{i\in I}\F_i^\ell,\otimes_{i\in I}\PP_i^\ell\right),$$
which belongs to $\mathcal{S}$. It is the probability space associated to $\hat{\be}_\ep^\ell$. Now, for any $\om^\ell=(\om^\ell_i)_{i\in I}\in \Om^\ell$ and $\uuu\in\U(t)$, let
\begin{equation*}
\hat{\beta}_\ep^\ell(\om^\ell, \uuu)(s)=
\begin{cases}
\be_{0}(\uuu)(s), & \text{ if } s\in [t,t+h],\\
\be^\ell_i(\om_i^\ell,\uuu'(s)), & \text{ if } s\in [t+h,1], \ \text{and} \ \xx[t,x,\uuu,\be_0](t+h)\in E_i,
\end{cases}
\end{equation*}
where $\uuu'$ denotes the restriction of $\uuu$ to $[t+h,1]$. Let $\al\in \A(t)$ be a strategy of player $1$, and let $(\uuu_0,\vvv_0) $ be (the unique pair) such that $\al(\vvv_0)=\uuu_0$ and $\be_0(\uuu_0)=\vvv_0$ a.e. on $[t,t+h]$. For any $\vvv'\in \V(t+h)$, let $\vvv' \circ \vvv_0\in \V(t)$ be the control obtained by concatenating $\vvv_0$ and $\vvv'$ at time $t+h$. Define $\al'\in \A(t+h)$ by the relation $\al'(\vvv'):=\al(\vvv'\circ \vvv_0)$. Now, by the choice of $\hat{\beta}_\ep$, the following relation holds
for any  $(k,\ell)\in K\times L$ and $\om^\ell\in \Om^\ell$:
\begin{equation}\label{rece}
g^{k\ell} (\xx[t,x,\al,\hat{\be}_\ep^\ell(\om^\ell)](1)=\sum_{i\in I}
g^{k\ell} \left(\xx\big[t+h,\XX_{t+h}^{t,x,\al,\be_0}, \al',\hat{\be}_i^\ell(\om^\ell_i)\big](1)\right)
\ind_{F_i},
\end{equation}
where $\XX_{t+h}^{t,x,\al,\be_0}=\xx[t,x,\al,\beta_0](t+h)$ and $F_i=\{\XX_{t+h}^{t,x,\al,\be_0}\in E_i\}$. \\
We claim that
\begin{equation}\label{ccclaim}\max_{k\in K}\EE^{k,Q}_{\hat{\beta}_\ep}
\bigg[g^{k\ell}\big(\xx[t,x,\al,\be(\om_2)](1)\big)\bigg]-\zeta^k \leq
\VV_K^{+\sharp}(t+h,\XX_{t+h}^{t,x,\al,\be_0},\zeta,Q)+2\ep.
\end{equation}
\emph{Proof of the claim.}
From \eqref{rece}, one deduces the
following equality:
$$\max_{k\in K}\EE^{k,Q}_{\hat{\beta}_\ep}
\bigg[g^{k\ell}\big(\xx[t,x,\al,\be(\om_2)](1)\big)\bigg]-\zeta^k=\quad\quad\quad\quad \quad  $$
\begin{equation}\label{dr1}
 \max_{k\in K}
 \sum_{i\in I}
g^{k\ell} \left(\xx\big[t+h,\XX_{t+h}^{t,x,\al,\be_0},\al', \hat{\be}_i^\ell(\om^\ell_i)\big](1)\right)
\ind_{F_i}-\zeta^k.
\end{equation}
The convexity of $z\mapsto \max_{k\in K}{z^k}$ implies that \eqref{dr1} is smaller than:
\begin{equation}\label{dr2}
  \sum_{i\in I} \sup_{\al'\in \A(t+h)}
\max_{k\in K}\bigg( g^{k\ell} \left(\xx\big[t+h,\XX_{t+h}^{t,x,\al,\be_0},\al', \hat{\be}_i^\ell(\om^\ell_i)\big](1)\right)-\zeta^k\bigg)\ind_{F_i}.
\end{equation}
By the choice of $\hat{\beta}_i$, \eqref{dr2} is smaller than
\begin{equation}\label{dr3}
 \sum_{i\in I} \bigg( \VV_K^{+\sharp}\big(t+h,\XX_{t+h}^{t,x,\al,\be_0},\zeta,Q\big)+2\ep\bigg) \ind_{F_i},
\end{equation}
which is equal to $\VV_K^{+\sharp}(t+h,\XX_{t+h}^{t,x,\al,\be_0},\zeta,Q)+2\ep$, so that the claim follows. \\
Taking the $\inf_{\hat{\be}\in \B_r(t)}\sup_{\al\in \A(t)}$ in both sides of \eqref{ccclaim}, one obtains
\begin{eqnarray*}\label{ccclaim2}\inf_{\hat{\be}\in \B_r(t)}\sup_{\al\in \A(t)}\max_{k\in K}\EE^{k,Q}_{\hat{\beta}_\ep}
\bigg[g^{k\ell}\big(\xx[t,x,\al,\be(\om_2)](1)\big)\bigg]-\zeta^k &\leq&
\sup_{\al\in \A(t)}
\VV_K^{+\sharp}(t+h,\XX_{t+h}^{t,x,\al,\be_0},\zeta,Q)+2\ep,\\
&\leq & \inf_{\be\in \B(t)}\sup_{\al\in \A(t)} \VV^{+\sharp}_K\big(t+h,\XX_{t+h}^{t,x,\al,\be},\zeta,Q\big)+3\ep.
\end{eqnarray*}
The first inequality holds because the right-hand-side of \eqref{ccclaim} does not depend on $\hat{\be}$.
The second, by the choice of $\be_0$. Together with \eqref{alt3}, these inequalities complete the proof.
\end{proof}
It is well-known that a function satisfying a sub-dynamic programming principle is a sub-solution
of the associated Hamilton-Jacobi equation.  we refer the reader to \cite{ES84} (resp. \cite{carda07})
for the case where the game is played with classical non-anticipative strategies (resp. non-anticipative with delay, or
discretized strategies).
\begin{corollaire} \label{corol1}For any $(\zeta,Q)\in \RR^K\times \De(L)^K$, the map $(t,x)\mapsto \VV_K^{+\sharp}(t,x,\zeta,Q)$ is a sub-solution of \eqref{HJI}.
\end{corollaire}
Reversing the roles of the players, one obtains the following result in the same manner.
\begin{corollaire} \label{corol2}For any $(P, \eta)\in \De(K)^L\times \RR^L$, the map $(t,x)\mapsto \VV_L^{-\flat}(t,x,P,\eta)$ is a super-solution of \eqref{HJI}.
\end{corollaire}

\subsection{Comparison principle}
Consider some general Hamilton-Jacobi equation, i.e. not necessarily depending on $f$:
\begin{equation}\label{genHJIJD}\tag{H}
\partial_t w(t,x)+H(t,x,Dw(t,x))=0, \quad \text{on } (0,1)\times \RR^n.
\end{equation}
Suppose that the Hamiltonian $H:[0,1]\times \RR^n\times \RR^n\to \RR$
is continuous and such that,
for all $t,s \in [0,1]$ and $x,y,\xi\in \RR^n$:
$$|H(t,x,\xi)-H(s,y,\xi)| \leq C\|\xi\|(|t-s|+\|x-y\|), \quad \text{for some }C\geq 0.$$

\begin{theoreme}[Comparison principle]\label{newcp'} Let $w^1, w^2: [0,1]\times \RR^{n}\times \De(K\times L)\to \RR$ be two Lipschitz continuous functions satisfying:
\begin{itemize}\item[-]  $w^1_K,\, w^2_K$ are concave on $\De(K)$ and $w^1_L,\, w^2_L$ are convex on $\De(L)$,
\item[-] For all $(\zeta,Q)\in \RR^K \times \De(L)^K$,
$(t,x)\mapsto w_K^{2\, \sharp}(t,x,\zeta,Q)$ is a sub-solution of \eqref{genHJIJD},
\item[-] For all $(P,\eta)\in \De(K)^L\times \RR^L$,
$(t,x)\mapsto w_L^{1 \, \flat}(t,x,P,\eta)$ is a super-solution of \eqref{genHJIJD},
\item[-]  $w^1(1,x,\pi)\geq w^2(1,x,\pi)$, for all $(x,\pi)\in \RR^N\times \De(K\times L)$.
\end{itemize}
Then $w^1(t,x,\pi)\geq w^2(t,x,\pi)$, for all $(t,x,\pi)\in [0,1]\times \RR^n\times \De(K \times L)$.
\end{theoreme}
This statement extends \cite[Theorem 5.1]{carda07} to the case of general type dependence.
The proof follows the same lines and is inspired in the proof of the Comparison Principle by Bardi and Capuzzo-Dolcetta \cite[Theorem 3.7]{BCD97}.
\begin{proof}
Suppose, on the contrary, the existence of some $(t',x', \pi')\in [0,1]\times \RR^n\times \De(K\times L)$ such that
$w^2(t',x',\pi')>w^1(t',x',\pi')$. Then, for some $\si >0$,
\[ \sup_{(t,x,\pi)}w^2(t,x,\pi)-w^1(t,x,\pi)-\si(1-t)>0,
\]
where the supremum is taken over $[0,1]\times \RR^n\times \De(K\times L)$.
Let us consider the standard method of separation of variables. To avoid technical details, we will assume that for some $R>0$,  $w^1(t,x,\pi)\geq w^2(t,x,\pi)$ for all $(t,x,\pi)$ with $\|x\|>R$. This assumption can be omitted by using penalization arguments at infinity (see \cite{BCD97}). Let $\ep>0$ be fixed and consider the following map:
\begin{equation}\label{esd1}
 (t,x,s,y,\pi)\mapsto w^2(s,y,\pi)-w^1(t,x,\pi)-\frac{1}{\ep}\|(s,y)-(t,x)\|^2-\si(1-t)>0.
\end{equation}
It attains its maximum at some point, denoted by $(t_\ep,x_\ep,s_\ep,y_\ep,\pi_\ep)$.
Let $p_\ep\in \De(K),\, q_\ep\in \De(L), \, Q_\ep\in \De(L)^K,\, P_\ep\in \De(K)^L$ b such that
$$\pi_\ep=p_\ep \otimes Q_\ep=q_\ep\otimes P_\ep.$$ From usual arguments (see \cite{BCD97}), $t_\ep,s_\ep<1$ for small $\ep$, because $w_2(1,y,\pi)-w_1(1,x,\pi)\leq 0$ and $w^1$ and $w^2$ are Lipschitz continuous. Moreover,
\[ \lim_{\ep \to 0^+}\frac{1}{\ep}\|(s_\ep,y_\ep)-(t_\ep,x_\ep)\|^2=0.
\]
\textbf{Part 1.} Here, we use that $w_L^{1\, \flat}$ is a super-solution to \eqref{genHJIJD}.\\
Fix $(s,y,P)=(s_\ep,y_\ep,P_\ep)$. Then, $(t_\ep,x_\ep,q_\ep)$ is a maximizer in \eqref{esd1} so that for all $(t,x,q)\in [0,1]\times \RR^n\times \De(L)$,
\begin{equation}\label{conj'}
 w_L^1(t,x,P_\ep,q)\geq w^1_L(t_\ep,x_\ep,P_\ep,q_\ep)+w^2_L(s_\ep,y_\ep,P_\ep,q)-w^2_L(s_\ep,y_\ep,P_\ep,q_\ep)+\varphi(t,x),
\end{equation}
where $\varphi(t,x)=\varphi_{(t_\ep,x_\ep,s_\ep,y_\ep)}(t,x)$ is defined as follows:
 \[\varphi(t,x):= \frac{1}{\ep}\big(\|(s_\ep,y_\ep)-(t_\ep,x_\ep)\|^2-
\| (s_\ep,y_\ep)-(t,x)\|^2 \big)+\si(t-t_\ep).\]
In particular, putting $(t,x)=(t_\ep,x_\ep)$ and because $\varphi(t_\ep,x_\ep)=0$, one has:
\begin{equation}\label{subs}
 w^1_L(t_\ep,x_\ep,P_\ep,q_\ep)-w^1_L(t_\ep,x_\ep,P_\ep,q)\leq w^2_L(s_\ep,y_\ep,P_\ep,q_\ep)-w^2_L(s_\ep,y_\ep,P_\ep,q).
\end{equation}
Let $\eta_\ep$ be in the sub-differential of the convex function
$q\mapsto w^2_L(s_\ep,y_\ep,P_\ep,q)$ at $q_\ep$, i.e.:
 $$w^2_L(s_\ep,y_\ep,P_\ep,q_\ep)+\langle \eta_\ep,q-q_\ep\rangle\leq  w^2_L(s_\ep,y_\ep,P_\ep,q), \ \forall q\in \De(L).$$
Then, by \eqref{subs},
$\eta_\ep\in \partial^-_q  w^1_L(t_\ep,x_\ep,P_\ep,q_\ep)$ too.
This implies (see Lemma \ref{fenchi}) that:
\begin{equation}\label{conj2'}
 \begin{aligned}
 w_L^{1\, \flat}(t_\ep,x_\ep,P_\ep,-\eta_\ep)&=&
  w^1_L(t_\ep,x_\ep,P_\ep,q_\ep)+ \langle -\eta_\ep, q_\ep\rangle,\\
w_L^{2\, \flat}(s_\ep,y_\ep,P_\ep,-\eta_\ep)&=&   w^2_L(s_\ep,y_\ep,P_\ep,q_\ep)+\langle -\eta_\ep, q_\ep\rangle.
\end{aligned}
\end{equation}
Taking the lower conjugate of \eqref{conj'} 
at $-\eta_\ep$, one obtains that
\begin{equation}
  w_L^{1\, \flat}(t,x,P_\ep,-\eta_\ep)\geq w^1_L(t_\ep,x_\ep,P_\ep,q_\ep)+ w_L^{2\, \flat}(s_\ep,y_\ep,P_\ep,-\eta_\ep)-w^2_L(s_\ep,y_\ep,P_\ep,q_\ep)+\varphi(t,x).
\end{equation}
 Using the equalities in \eqref{conj2'}, we obtain that
\begin{equation}
 w_L^{1\, \flat}(t,x,P_\ep,-\eta_\ep)\geq  w_L^{1\, \flat}(t_\ep,x_\ep,P_\ep,-\eta_\ep)+\varphi(t,x),
\end{equation}
with an equality at $(t,x)=(t_\ep,x_\ep)$. The right-hand-side is a $\mathcal{C}^1$-function, and can be taken as a test function. Its derivatives are clearly those of
$\varphi$, i.e.:
$$ \partial_t \varphi (t_\ep,x_\ep)=\si+\frac{2}{\ep}(s_\ep-t_\ep),$$
$$ D \varphi (t_\ep,x_\ep)=\frac{2}{\ep}(y_\ep-x_\ep).$$
Finally, $(t,x)\mapsto w_L^{1\,\flat}(t,x,P_\ep,-\eta_\ep)$ being a super-solution of \eqref{genHJIJD}, one has:
\begin{equation}\label{aa1'}
\si+\frac{2}{\ep}(s_\ep-t_\ep)+H(t_\ep,x_\ep,\frac{2}{\ep}(y_\ep-x_\ep))\leq 0.
\end{equation}
\textbf{Part 2.} Symmetrically, we use here that $w_K^{2\, \sharp}$ is a sub-solution to \eqref{genHJIJD}. We fix
$(t,x,Q)=(t_\ep,x_\ep,Q_\ep)$ and let $(s_\ep,y_\ep,p_\ep)$ be a maximizer in \eqref{esd1}.
As in Part 1, one obtains
\begin{equation*}
 w_K^{2\, \sharp}(s,y,\zeta_\ep,Q_\ep)\leq
 w_K^{2\, \sharp}(s_\ep,y_\ep,\zeta_\ep,Q_\ep )+\phi(s,y),
\end{equation*}
with an equality at $(s_\ep,y_\ep)$, and where $\phi$ is 
a $\mathcal{C}^1$ and satisfies: 
$$ \partial_t \phi (t_\ep,x_\ep)=\frac{2}{\ep}(s_\ep-t_\ep),$$
$$ D \phi (t_\ep,x_\ep)=\frac{2}{\ep}(y_\ep-x_\ep).$$
The mapping $(t,x)\mapsto w_K^{2\, \sharp}(t,x, \zeta_\ep,Q_\ep)$ being a sub-solution of \eqref{genHJIJD}, and using $\phi$ as a test function one obtains:
\begin{equation}\label{aa2'}
\frac{2}{\ep}(s_\ep-t_\ep)+H(s_\ep,y_\ep,\frac{2}{\ep}(y_\ep-x_\ep))\geq 0.
\end{equation}
\textbf{Conclusion.}
Subtracting
\eqref{aa1'} and \eqref{aa2'} yields:
\begin{equation}
H(t_\ep,x_\ep,\frac{2}{\ep}(y_\ep-x_\ep))-H(s_\ep,y_\ep,\frac{2}{\ep}(y_\ep-x_\ep))\leq -\si.
\end{equation}
Using the assumptions on the Hamiltonian, one gets a contradiction by letting $\ep\to 0$.
\end{proof}
\begin{proof}[\textbf{Proof of Theorem \ref{main1}}] $\phantom{Helena}$\\
Let $w^1=\VV^-$ and $w^2=\VV^+$. These  functions satisfy the assumptions of the comparison principle (Theorem \ref{newcp'}). Indeed, we already noticed that the concavity-convexity is a general property for games with incomplete information. Corollary \ref{corol1} (resp. \ref{corol2}) gives the second (resp. third) assumption. Finally, by definition, one has
$$\VV^-(1,x,\pi)=\VV^+(1,x,\pi)=\sum_{(k,\ell)\in K\times L}\pi^{k\ell} g^{k\ell}(x), \quad \forall (x,\pi)\in \RR^n\times \De(K\times L).$$
The fact that the Hamiltonian associated to the game $H=H^+=H^-$ satisfies the assumptions in the comparison principle, follows from the Assumption \ref{ass1JD} and from Cauchy-Schwartz inequality. Indeed, there exists $c\geq 0$ such that,
for any $(u,v)\in U\times V$, $s,t\in [0,1]$, $x,y,\xi,\in \RR^n$:
\begin{equation}\label{ineq}
\begin{array}{l c l}
\left | \langle f(t,x,u,v),\xi\rangle -\langle f(s,y,u,v),\xi\rangle \right |&\leq&
\|\xi\| \| f(t,x,u,v)-f(s,y,u,v)\|,\\&\leq &
\|\xi\| c\left (|t-s|+\| x-y\|\right).
\end{array}
\end{equation}
A  classical $\sup \inf$ argument gives then
$$\left | H(t,x,\xi)-H(t,y,\xi)\right | \leq
\|\xi\| c\left (|t-s|+\| x-y\|\right).$$
Theorem \ref{newcp'} applies, so that $\VV^-\geq \VV^+$.
\end{proof}

\subsection{Application to repeated games}
In this paragraph, we deduce from Theorem \ref{main1} a new characterization for the asymptotic value of repeated games with incomplete information.

Let 
$I$ and $J$ be two finite sets, let $U:=\De(I)$ and $V:=\De(J)$ stand for the corresponding simplexes and let $G^{k\ell}:I\times J\to \RR$ be a matrix game for each $(k,\ell)\in K\times L$. Let $\pi\in \De(K\times L)$ be an initial probability and let $\theta\in \De(\NN^*)$ be a measure giving the weight of each stage. 
A repeated game with incomplete information is played as follows: 
\begin{itemize}
\item[-]
First, a pair of signals $(k,\ell)$ is drawn according to $\pi$; player $1$ is informed of $k$, player $2$ of $\ell$.
\item[-] Then, at every stage $m\geq 1$, knowing the past actions, the players choose actions $(i_m,j_m)\in I\times J$.
\end{itemize}
Player $1$ maximizes $\sum_{m\geq 1}\theta_m G^{k\ell}(i_m,j_m)$. The existence of the value $v_\theta(\pi)$ is straightforward. The convergence of $v_\theta(\pi)$, as $\|\theta\|:=\max_{m\geq 1}\theta_m$ tends to $0$, was established by Mertens and Zamir \cite{MZ71}, for the two classical evaluations $\theta_m=\frac{1}{n}\ind_{m\geq n}$ and $\theta_m=\la(1-\la)^{m-1}$, extended to a general evaluation by Cardaliaguet, Laraki and Sorin \cite{CLS11}, in the independent case, and then to the general case by Oliu-Barton \cite[Section 5.4]{OB13}.
The limit, denoted by $v(\, \cdot \,)$, is the
unique solution of the Mertens and Zamir \cite{MZ71} system of functional equations on $\De(K\times L)$: 
\begin{equation}\label{MZ2}\tag{MZ}
\begin{cases}
\begin{array}{r c l} w(\pi)&=&\mathrm{Cav}_{\De(K)}\min\{u_K,w_K\}(\pi^K,\pi^{L|K})\\
w(\pi)&=&\mathrm{Vex}_{\De(L)}\max\{u_L,w_L\}(\pi^L,\pi^{K|L})
\end{array} \end{cases}
\end{equation}
where $u:\De(K\times L)\to \RR$ is the value of the \emph{average} (or non-revealing) game:  $$u(\pi)=\mathrm{val}_{(u,v)\in U\times V} \sum_{(k,\ell)\in K\times L} \pi^{k\ell} G^{k\ell}(u,v),$$
which exists by the minmax theorem. Here, $G^{k\ell}$ is bi-linearly extended to $U\times V$.

Consider now a natural continuous-time analog of the game we just described:
\begin{itemize}
\item[-] First, a pair of signals $(k,\ell)$ is drawn according to $\pi\in \De(K\times L)$, player $1$ is informed of $k$, player $2$ of $\ell$.
\item[-] Then, the players play a differential game with 
initial time $t_0$, running payoff $G^{k\ell}$ and no dynamic, i.e. there is no state variable.
\end{itemize}
By stationarity (i.e. there is no dependence on $t$) 
one can assume w.l.o.g. that $t_0=0$ and thus use the shorter notation
$\A$,  $\A_r$, $\B$ and $\B_r$ for $\A(0)$, $\A_r(0)$, $\B(0)$ and $\B_r(0)$ respectively.
The value exists thanks to Theorem \ref{main1}:
$$\WW(\pi):=\mathrm{val}_{(\al,\be)\in \A_r^K\times \B_r^L} \EE_{\hat{\al},\hat{\beta}}^{\pi}\left[ \int_{0}^{1} G^{k\ell}(\uuu_\om(s),\vvv_\om(s))ds\right].$$

The characterization of $\WW$ will yield a new characterization of
$v(\, \cdot \,)$, together with the equality $\WW(\pi)=v(\pi)$ for all $\pi\in \De(K\times L)$.
Start by using he reduction of Section \ref{reduc} (see the Appendix). Define an auxiliary game with initial state $x_0\in \RR^{K\times L}$, dynamic $f(t,x,u,v):=(G^{k\ell}(u,v))_{k\ell}$, no running payoff and terminal payoff functions $g^{k\ell}(x)=x^{k\ell}$, for each $(k,\ell)\in K\times L$. The state accounts  for the cumulated payoff in each one of the $K\times L$ coordinates:
$$x_t=\int_{t_0}^t  G(\uuu(s),\vvv(s))ds\in \RR^{K\times L},$$
and the terminal payoff is simply the cumulated payoff corresponding to the true parameters. 
 This game has a value $\VV(t_0,x_0,\pi)$ by Theorem \ref{main1}. The following properties are straightforward consequences of the stationarity of the model: for all $(t,x,\pi)\in [0,1]\times \RR^{K\times L}\times \De(K\times L)$,
\begin{equation}\label{properties}
\begin{array}{r c l}
\VV(t,x,\pi)&=&\VV(t,0,\pi)+\langle x,\pi\rangle,\\
\VV(t,x,\pi)&=&(1-t)\VV(0,x,\pi),\\
\VV(0,0,\pi)&=&\WW(\pi).
\end{array}
\end{equation}
The next result is a consequence  of the the sub-dynamic programming principle (Proposition  \ref{subprogr'}), the characterization in Theorem \ref{main1} and the equalities \eqref{properties}.
\begin{proposition}\label{caracbis}
The map $\WW:\De(K\times L)\to \RR$
is the unique $K$-concave and $L$-convex Lipschitz continuous
function satisfying: 
\begin{enumerate}
\item[$(1)$] For all $Q\in \De(L)^K$, $\zeta\mapsto \WW_K^\sharp(\zeta, Q)$ is a sub-solution of:
$$-w(\zeta)+\langle D w(\zeta),\zeta\rangle+u_K(-Dw(\zeta),Q)\geq 0, \quad \text{on } \ \RR^K.$$
\item[$(2)$]  For all $P\in \De(K)^L$, $\eta\mapsto \WW_L^\flat(P,\eta)$ is a super-solution of:
$$-w(\eta)+\langle D w(\eta),\eta\rangle +u_L( P,Dw(\eta))\leq 0, \quad \text{on } \ \RR^L.$$
\end{enumerate}
\end{proposition}
\begin{proof}
The $K$-concavity, $L$-convexity and Lipschitz continuity of $\WW$ is a direct consequence of Theorem \ref{main1} and the relation $\VV(0,0,\pi)=\WW(\pi)$.
Now, it follows from \eqref{properties} that, for all $(t,x,\zeta,Q)\in [0,1)\times \RR^n\times \RR^K\times \De(L)^K$,
\begin{equation}\label{properties2}
\VV_K^{\sharp}(t,x,\zeta,Q)=(1-t)\WW_K^\sharp\left(\frac{\zeta-x(Q)}{1-t}, Q\right),
\end{equation} 
where $x^k(Q)=\sum_{\ell\in L}Q(\ell|k)x^{k\ell}$, for all $k\in K$.
Applying Proposition \ref{subprogr'} at $(t,x)=(0,0)$, one obtains that, for all $(\zeta,Q)\in \RR^K\times \De(L)^K$ and $h\in [0,1)$,
\begin{equation}\label{subprogrw}
\WW_K^\sharp(\zeta,Q)\leq (1-h){}{} \inf_{\be\in \B}\sup_{\al\in \A} \WW_K^\sharp\left(\frac{\zeta-G^Q_h(\al,\be)}{1-h}, Q\right),
\end{equation}
where $G^Q_h(\al,\be)\in \RR^K$, 
$$(G^Q_h(\al,\be))^k =  \int_0^h \sum\nolimits_{\ell\in L} Q(\ell|k)G^{k\ell}(\al(s),\be(s))ds.$$ 
By classical arguments, the super-dynamic principle \eqref{subprogrw} implies $(1)$.
Reversing the roles of the players, one obtains $(2)$.
Finally, uniqueness follows from a standard comparison principle. 
\end{proof}
It is somehow deceiving to characterize $\WW$ in terms of its upper and lower conjugates. The next results shows how, from Proposition \ref{caracbis}, one can recover properties on the original function.
\begin{corollaire}\label{corol} Let $Q_0\in \De(L)^K$ be fixed, and suppose that $\zeta\mapsto \WW_K^\sharp(\zeta, Q_0)$ is differentiable at $\zeta_0\in \RR^K$. Then,
$$\WW(p_0\otimes Q_0)\leq u(p_0\otimes Q_0), \quad \text{where}\quad p_0:=- D \WW_K^\sharp(\zeta_0,Q_0)\in \De(K).$$
\end{corollaire}
\begin{proof}
Let $\phi$ be a test function at $\zeta_0$, i.e. $\phi(\zeta)\geq \WW_K^\sharp(\zeta, Q_0)$, for all $\zeta\in \RR^K$, with an equality at $\zeta_0$. The differentiability of
$\zeta\mapsto \WW_K^\sharp(\zeta, Q_0)$ at $\zeta_0$ implies then that
$D \phi(\zeta_0)=D \WW_K^\sharp(\zeta_0, Q_0)$.
Then, by Proposition \ref{caracbis} and the choice of $p_0$ one has:
\begin{equation}\label{vi}
 -\WW^\sharp_K(\zeta_0,Q_0)+\langle -p_0,\zeta_0 \rangle + u_K(p_0,Q_0)\geq 0.
\end{equation}
Finally, by Fenchel duality (see Lemma \ref{fenchi}) one has
$$\WW^\sharp_K(\zeta_0,Q_0)=-\WW_K(p_0,Q_0)+\langle p_0,\zeta_0 \rangle.$$
Replacing this expression in \eqref{vi} gives the desired result.
\end{proof}

We can now obtain the desired equality.

\begin{corollaire}\label{equal} $\WW(\pi)=v(\pi)$ for all $\pi\in \De(K\times L)$
\end{corollaire}
\begin{proof}
By Lemma \ref{envelope}, the differentiability of $\zeta\mapsto \WW_K^\sharp(\zeta, Q_0)$ at $\zeta_0\in \RR^K$ is equivalent to the fact that $p_0:=- D \WW_K^\sharp(\zeta_0,Q_0)\in \De(K)$ is an extreme point of $p\mapsto \WW_K(p, Q_0)$.
On the other hand, any $K$-concave, $L$-convex continuous, bounded function $f:\De(K\times L)\to \RR$ is a solution to \eqref{MZ2} if and only if (see \cite[Lemma 4.35]{sorin02})
\begin{equation*}
\begin{cases}
\begin{array}{r c l r} f(p\otimes Q)&\leq &u(p\otimes Q), & \forall p\in \mathcal{E}_f(Q), \\
f(q\otimes P)&\geq &u(q\otimes P), & \forall q\in \mathcal{E}_f(P),\end{array} \end{cases}
\end{equation*}
where $\mathcal{E}_f(Q)$ (resp.  $\mathcal{E}_f(P)$) is the set of extreme points of  $f_K(\, \cdot \, Q)$ (resp.
$f_L(\, \cdot \, P)$). By Corollary \ref{corol}, $\WW$ is thus a solution of the Mertens-Zamir system of functional equations. The solution being unique \cite[Theorem 2.1]{MZ71}, one obtains $\WW=v$.
\end{proof}

Corollary \ref{equal} and Proposition \ref{caracbis} provide a new  characterization for the asymptotic value $v(\, \cdot \,)$ of repeated games with incomplete information.


 \section*{Acknowledgements}
 The author is grateful to Sylvain Sorin for his reading and remarks.
 He is also very much indebted to Pierre Cardaliaguet and Catherine Rainer for
 their lectures on the subject, and their encouraging feedback.
The author is also very thankful to the two anonymous referees, whose comments have helped improving this paper.

\section*{Appendix}
Let us describe precisely the standard transformation from a Bolza to a Mayer problem, which allows to assume without loss of generality, that there is no running payoff.
The past controls being commonly observed,
both players can compute the $K\times L$ potential integral payoffs and positions induced by the pair $(\uuu,\vvv)\in \U(t_0)\times \V(t_0)$ at time $t\in [t_0,1]$: 
$$
\int_{t_0}^t \ga^{k\ell}(s,\xx^{k\ell}[t_0,x^{k\ell}_0,\uuu,\vvv](s),\uuu(s),\vvv(s))ds, \quad (k,\ell)\in K\times L$$
$$\xx^{k\ell}[t_0,x^{k\ell}_0,\uuu,\vvv](t), \quad (k,\ell)\in K\times L.
$$
Define a new state variable in $(\RR\times \RR^n)^{K\times L }$ which contains this information. Let the dynamic be given by:
\begin{eqnarray*}
F:[0,1]\times (\RR\times \RR^n)^{K\times L}\times U\times V&\to & (\RR\times \RR^n)^{K\times L},\\
F^{k\ell}\left(t,(y^{k\ell},x^{k\ell})_{(k,\ell)},u,v\right)&=&
\left(\ga^{k\ell} (t,x^{k\ell},u,v ), f^{k\ell}(t,x^{k\ell},u,v)\right). \phantom{\sum_p^t}
\end{eqnarray*}
Define new terminal payoff functions by setting, for each $(k,\ell)\in K\times L$, 
\begin{eqnarray*}
G^{k\ell}:(\RR\times \RR^n)^{K\times L}&\to & \RR,\\
G^{k\ell}((y^{k,\ell},x^{k\ell})_{(k,\ell)})&=& y^{k\ell}+g^{k\ell}(x^{k\ell}). \phantom{\sum_p^t}
\end{eqnarray*}
Let $N=(1+n)|K| |L|$ and let $X_0:=(0,x_0^{k\ell})_{(k,\ell)\in K\times L}\in \RR^N$.
The regularity of $f$, $\ga$ and $g$ ensures that $F$ and $G^{k\ell}$ satisfy Assumption \ref{ass1JD},
for each $(k,\ell)$. Define an auxiliary differential game with asymmetric information 
as follows:
before the game starts, $(k,\ell)\in K\times L$ is drawn according to $\pi$; $k$ (resp. $\ell$) is told to player $1$ (resp. $2$). Then,
the standard differential game
$$(X_0,F,0,(G^{k\ell})_{k,\ell})$$
is played. This game is equivalent to $\G(t_0,\pi)$ since any couple of  controls $(\uuu,\vvv)\in \U(t_0)\times \V(t_0)$ induces the same payoff in both games.

In view of this reduction, with no loss of generality we may focus on games satisfying $(a)$ and $(b)$.
The following assumption holds in the rest of the paper.

\begin{ass}\label{ass3JD} There exists $z\in \RR^n$ and $\phi:[0,1]\times \RR^n \times U\times V\to \RR^n$ such that
$x_0^{k\ell}=z$ and $f^{k\ell}=\phi$, for all $(k,\ell)\in K\times L$. Moreover, $\ga\equiv 0$.
 \end{ass}

  \bibliographystyle{amsplain}
 \bibliography{bibliothese2}

\providecommand{\bysame}{\leavevmode\hbox to3em{\hrulefill}\thinspace}
\providecommand{\MR}{\relax\ifhmode\unskip\space\fi MR }
\providecommand{\MRhref}[2]{%
  \href{http://www.ams.org/mathscinet-getitem?mr=#1}{#2}
}
\providecommand{\href}[2]{#2}
\begin{thebibliography}{10}

\bibitem{AM95}
R.J. Aumann and M.~Maschler, \emph{Repeated games with incomplete information,
  \emph{with the collaboration of R. Stearns}}, MIT Press, 1995.

\bibitem{BCD97}
M.~Bardi and I.~Capuzzo-Dolcetta, \emph{{O}ptimal {C}ontrol and {V}iscosity
  {S}olutions of {H}amilton-{J}acobi-{B}ellman equations}, Birkh\"{a}user,
  1997.

\bibitem{cardanotes}
P.~Cardaliaguet, \emph{Introduction to differential games. \emph{{L}ecture
  Notes, {U}niversit\'e de {B}retagne {O}ccidentale}, 2010.}

\bibitem{carda07}
\bysame, \emph{Differential games with asymmetric information}, SIAM journal on
  Control and Optimization \textbf{46} (2007), 816--838.

\bibitem{carda09b}
\bysame, \emph{A double obstacle problem arising in differential game theory},
  Journal of Mathematical Analysis and Applications \textbf{360} (2009),
  95--107.

\bibitem{carda09a}
\bysame, \emph{Numerical approximation and optimal strategies for differential
  games with lack of information on one side}, Advances in Dynamic Games and
  their Applications \textbf{10} (2009), 159--176.

\bibitem{CLS11}
P.~Cardaliaguet, R.~Laraki, and S.~Sorin, \emph{A continuous time approach for
  the asymptotic value in two-person zero-sum repeated games}, SIAM Journal on
  Control and Optimization \textbf{50} (2012), 1573--1596.

\bibitem{CQ08}
P.~Cardaliaguet and M.~Quincampoix, \emph{Deterministic differential games
  under probability knowledge of initial condition}, International Game Theory
  Review \textbf{10} (2008), 1--16.

\bibitem{CR09}
P.~Cardaliaguet and C.~Rainer, \emph{On a continuous-time game with incomplete
  information}, Mathematics of Operations Research \textbf{34} (2009),
  769--794.

\bibitem{CR09b}
\bysame, \emph{Stochastic differential games with asymmetric information},
  Applied Mathematics \& Optimization \textbf{59} (2009), 1--36.

\bibitem{CS10}
P.~Cardaliaguet and A.~Souqui{\`e}re, \emph{A differential game with a blind
  player}, SIAM Journal on Control and Optimization \textbf{50} (2012),
  2090--2116.

\bibitem{CL83}
M.G. Crandall and P.L. Lions, \emph{Viscosity solutions of {H}amilton-{J}acobi
  equations.}, Trans. Amer. Math. Soc. \textbf{277} (1983), 1--42.

\bibitem{dm96}
B.~De~Meyer, \emph{Repeated games and partial differential equations},
  Mathematics of Operations Research \textbf{21} (1996), 209--236.

\bibitem{ES84}
L.C. Evans and P.E. Souganidis, \emph{Differential games and representation
  formulas for solutions of {H}amilton-{J}acobi {E}quations}, Indiana Univ.
  Math. J. \textbf{282} (1984), 487--502.

\bibitem{GOB13}
F.~Gensbittel and M.~Oliu-Barton, \emph{Optimal strategies in repeated games
  with incomplete information}, Working paper (2013).

\bibitem{heuer92}
M.~Heuer, \emph{Asymptotically optimal strategies in repeated games with
  incomplete information}, International Journal of Game Theory \textbf{20}
  (1992), 377--392.

\bibitem{MZ71}
J.-F. Mertens and S.~Zamir, \emph{The value of two-person zero-sum repeated
  games with lack of information on both sides}, International Journal of Game
  Theory \textbf{1} (1971), 39--64.

\bibitem{OB13}
M.~Oliu-Barton, \emph{Dynamic games with incomplete information in discrete and
  continuous time}, PhD thesis, \emph{Universit\'e Pierre et Marie Curie}
  (2013).

\bibitem{rockafellar97}
R.T. Rockafellar, \emph{Convex {A}nalysis}, Princeton University Press, 1997.

\bibitem{sorin02}
S.~Sorin, \emph{A {F}irst {C}ourse on {Z}ero-{S}um {R}epeated {G}ames},
  Springer, 2002.

\bibitem{SZ85}
S.~Sorin and S.~Zamir, \emph{A 2-person game with lack of information on 1 and
  1/2 sides}, Mathematics of {O}perations {R}esearch \textbf{10} (1985),
  17--23.

\bibitem{souquiere10b}
A.~Souqui{\`e}re, \emph{Approximation and representation of the value for some
  differential games with imperfect information}, International Journal of Game
  Theory \textbf{39} (2010), 699--722.

\bibitem{souquiere10a}
\bysame, \emph{Jeux diff\'erentiels \`a information imparfaite}, PhD thesis,
  \emph{Universit\'e de Bretagne Occidentale} (2010).

\end{thebibliography}

\end{document}